%% file: main.tex
\documentclass{amsart}

\usepackage[utf8]{inputenc}
\usepackage{amsmath,amsthm, amssymb, amsfonts}
\usepackage{todonotes}
\usepackage{fullpage}
\usepackage{comment}
\usepackage{hyperref}
\usepackage{dsfont}

\usepackage{epigraph}
\usepackage{wasysym}
\usepackage{units}
\usepackage{caption}

\usepackage{stmaryrd}
\usepackage{xcolor}

\let\emptyset\varnothing
\newcommand{\Z}{\mathbb{Z}}
\newcommand{\N}{\mathbb{N}}
\newcommand{\R}{\mathbb{R}}
\newcommand{\Gr}{\mathcal{G}}

\newcommand{\eps}{\varepsilon}
\newcommand{\lift}[1]{#1^{\uparrow}}
\newcommand{\vect}[2]{#1^{\mathrm{vec},#2}}
\newcommand{\conv}[1]{\mathrm{Conv}(\Gr,#1)}
\newcommand{\filter}[1]{\mathrm{filter}(\Gr,#1)}
\newcommand{\equi}[1]{#1^{\mathrm{equi}}}
\newcommand{\arch}[1]{\mathcal{A}(#1)}
\newcommand{\FirstN}[1]{\llbracket #1 \rrbracket}
\newcommand{\Indicator}{{\mathds{1}}}

\newtheorem{theorem}{Theorem}[section]
\newtheorem{remark}[theorem]{Remark}
\newtheorem*{rem*}{Remark}
\newtheorem{definition}[theorem]{Definition}
\newtheorem{proposition}[theorem]{Proposition}

\title[Equivalence of approximation properties of CNNs and FNNs]
      {\texorpdfstring{Equivalence of approximation\\
                       by convolutional neural networks\\
                       and fully-connected networks}
                      {Equivalence of approximation
                       by convolutional neural networks
                       and fully-connected networks}}

\begin{document}
\date{}

\author{Philipp Petersen$^\dagger$}
\address[P.~Petersen]{Mathematical Institute,
University of Oxford,
Andrew Wiles Building, Oxford, UK}
\email{Philipp.Petersen@maths.ox.ac.uk}
\author{Felix Voigtlaender$^\dagger$}
\address[F.~Voigtlaender]{Lehrstuhl Wissenschaftliches Rechnen,
Katholische Universität Eichstätt--Ingolstadt,
Ostenstraße 26,
85072 Eichstätt,
Germany}
\email{felix@voigtlaender.xyz}
\thanks{$\dagger$ Both authors contributed equally to this work}




\begin{abstract}
Convolutional neural networks are the most widely used
type of neural networks in applications.
In mathematical analysis, however, mostly fully-connected networks are studied.
In this paper, we establish a connection between both network architectures.
Using this connection, we show that all upper and lower bounds
concerning approximation rates of {fully-connected} neural networks
for functions $f \in \mathcal{C}$---for an arbitrary function class
$\mathcal{C}$---translate to essentially the same bounds
concerning approximation rates of \emph{convolutional} neural networks for functions
$f \in \equi{\mathcal{C}}$, with the class $\equi{\mathcal{C}}$ consisting of all
translation equivariant functions whose first
coordinate belongs to $\mathcal{C}$.
All presented results consider exclusively the case of convolutional neural networks
without any pooling operation and with circular convolutions, i.e., not based on zero-padding.
\end{abstract}

\maketitle



\section{Introduction}

The recent overwhelming success of machine learning techniques such as
\emph{deep learning} \cite{Goodfellow-et-al-2016, LeCun2015DeepLearning, schmidhuber2015deep}
has prompted many theoretical works trying to provide a mathematical explanation
for this extraordinary performance.
One line of research focuses on analysing the underlying
computational architecture that is given by a neural network.
In the context of approximation theory, it is possible to
describe the capabilities of this architecture meaningfully.
First and foremost, the universal approximation theorem
(see \cite{Cybenko1989, Hornik1989universalApprox, PinkusUniversalApproximation})
shows that any continuous function on a compact domain can be approximated arbitrarily
well by neural networks.
Besides, more refined approximation results relate the size
(in terms of number of neurons or number of free parameters)
of an approximating neural network to its approximation fidelity;
see for instance \cite{Barron1993,Mhaskar1993, mhaskar-micchelli,PetV2018OptApproxReLU,
pinkus1999approximation,ShaCC2015provableAppDNN,YAROTSKY2017103}.
These results include upper bounds on the sufficient size of a network,
but also establish lower bounds on the necessary size of a network
which is required for certain approximation tasks.

While offering valuable insight into the functionality and capability of neural networks,
the practical relevance of these results is limited.
Indeed, all mentioned results consider so-called \emph{fully-connected neural networks} (FNNs).
In most applications, however, so-called \emph{convolutional neural networks} (CNNs),
\cite{LeCun}, are employed.

\medskip{}

We establish a transference result between CNNs and FNNs.
Concretely, we demonstrate that every FNN can be transformed into a CNN
with a comparable number of free parameters and vice versa.
This demonstrates that---at least from an approximation theoretical point of view---%
FNNs and CNNs can be considered equivalent.

\subsection{{Our contribution}}
\label{sub:SimplifiedResults}

\input{Nutshell.tex}

\subsection{Related work}
\label{sub:RelatedWork}

\input{RelatedWork.tex}


\subsection{Pooling}

The convolutional networks used in practice often employ a form of \emph{pooling} after
each layer (see \cite[Section 9.3]{Goodfellow-et-al-2016});
besides, the convolutions are sometimes \emph{zero-padded}
(see \cite[Section 9.5]{Goodfellow-et-al-2016}) instead of periodic.
However, both of these techniques destroy the translation equivariance.
For this reason---and for the sake of mathematical simplicity---we restrict ourselves to
the case of periodic convolutions without pooling in this short note.

\subsection*{Structure of the paper}

The paper is structured as follows:
We begin by introducing FNNs and CNNs in Sections~\ref{sec:FullyConnectedNetworks}
and \ref{sec:ConvNets}, respectively.
The aforementioned equivalence of approximation rates is then
demonstrated in Section~\ref{sec:ApproximationUsingCNN}.

\section{Fully-connected neural networks}
\label{sec:FullyConnectedNetworks}

Let $\Gr$ be a finite group of cardinality $|\Gr| \in \N$.
For a finite set $I$, we denote the set of real-valued sequences with index set $I$ by
$\R^{I} = \{ (x_i)_{i \in I} \,|\, x_i \in \R \,\, \forall \, i \in I\}$.
In this note, we consider neural network functions
the input of which are elements of $\R^{\FirstN{M} \times \Gr}$,
where $\FirstN{M} := \{1,\dots,M\}$.
In Section~\ref{sec:ApproximationUsingCNN},
we will compare the expressivity of such FNNs with that of CNNs.
Even though the group structure of $\Gr$ is not used in the present
section, it will be essential for defining CNNs in the next section.

The following definition of FNNs
is standard in the mathematical literature on neural networks;
only the restriction to inputs in $\R^{\FirstN{C_0} \times \Gr}$
is slightly unusual.

\begin{definition}\label{def:FullyConnectedNetworks}
Let $\Gr$ be a finite group, let $C_0, L\in \N$, and $N_1, \dots, N_L\in \N$.
A \emph{fully-connected neural network}  $\Phi = (V_1, \dots, V_L)$ is a sequence of affine-linear
maps,
where $V_1: \R^{\FirstN{C_0} \times \Gr} \to \R^{N_1}$
and ${V_\ell: \R^{N_{\ell-1}} \to \R^{N_\ell}}$ for $1 < \ell \leq L$.
The \emph{architecture} $\arch{\Phi}$ is given by
$\arch {\Phi} := (C_0 \cdot |\Gr|, N_1, \dots, N_L)$.
For an arbitrary function $\varrho:\R \to \R$
(called the \emph{activation function}), we define
the \emph{$\varrho$-realisation} of the network $\Phi = (V_1,\dots,V_L)$ as
\begin{align*}
  R_\varrho (\Phi) : \R^{\FirstN{C_0} \times \Gr} \to \R^{N_L} \, , \quad
   \,\, x &\mapsto x_L,
\end{align*}
where $x_0 := x$ and $x_{\ell + 1} := \varrho \big( V_{\ell+1} (x_\ell) \big)$ for $0 \leq \ell \leq L - 2$,
while $x_L := V_L (x_{L-1})$.
Here, $\varrho$ is applied component-wise, that is,
$\varrho \big( (x_1, \dots, x_n) \big) = \big( \varrho(x_1), \dots, \varrho(x_n) \big)$.

For an affine-linear map $V : \R^I \to \R^J$, there is a uniquely determined
vector $b \in \R^J$ and a linear map $A : \R^I \to \R^J$ such that
$V(x) = A \, x + b$ for all $x \in \R^I$.
We then set $\|V\|_{\ell^0} := \|b\|_{\ell^0} + \|A\|_{\ell^0}$, where
$\|b\|_{\ell^0} = |\{j \in J \,|\, b_j \neq 0\}|$ denotes the number of non-zero
entries of $b$, and where $\|A\|_{\ell^0} := \sum_{i \in I} \|A \, \delta_i \|_{\ell^0}$,
with $(\delta_i)_{i \in I}$ denoting the standard basis of $\R^I$.
With this notation, we define the \emph{number of weights} $W(\Phi)$ and the
\emph{number of neurons} $N(\Phi)$ as
\[
  W(\Phi) := \sum_{\ell=1}^L \|V_\ell\|_{\ell^0}
  \qquad \text{and} \qquad
  N(\Phi) := C_0 \cdot |\Gr| + \sum_{\ell=1}^L N_\ell \, .
\]
\end{definition}

\section{Convolutional neural networks}\label{sec:ConvNets}



For a finite group $\Gr$ and functions $a,b : \Gr \to \R$, we denote by
$a \ast b$ the \emph{convolution} of $a$ and $b$, defined by
\begin{equation}
  a \ast b : \Gr \to \R, \quad
  g \mapsto \sum_{h \in \Gr}
              a(h) \, b(h^{-1} g) \, .
  \label{eq:ConvolutionDefinition}
\end{equation}
The first step in computing the output of a CNN is to convolve the
input with different convolution kernels.
This leads to different channels, each of which has the same dimension.
Each layer of the network thus has a \emph{spatial dimension}
(the number of elements $|\Gr|$ of the group $\Gr$)
and a \emph{channel dimension} (the number of channels).
After the convolution step, the different channels are combined {in an affine-linear fashion},
but only along fixed spatial coordinates (see Figure~\ref{fig:CNNExplanation}).
Finally, the activation function is applied component-wise,
and the whole procedure is repeated on the next layer,
with input given by the output of the present layer.

We shall now turn this informal description into a formal definition.
The definitions might appear to be overly technical,
but these technicalities will be important later to estimate the number of parameters of a CNN.
Before stating that definition, we introduce notation allowing for more
succinct expressions.
For ${x = (x_{i,g})_{i \in \FirstN{M}, g \in \Gr} \in \R^{\FirstN{M} \times \Gr}}$,
we write $x_i := (x_{i,g})_{g \in \Gr} \in \R^{\Gr}$ for $i \in \FirstN{M}$.
Likewise, we will identify a family
$(x_i)_{i \in \FirstN{M}}$, where $x_i \in \R^{\Gr}$,
with the family $(x_i (g))_{i \in \FirstN{M}, g \in \Gr} \in \R^{\FirstN{M} \times \Gr}$.

Finally, if $I,J$ are sets, and if $F : \R^J \to \R^I$, then we define the \emph{lifting} of $F$
as the map $\lift{F}: \R^{J \times \Gr} \to \R^{I \times \Gr}$ that results from applying $F$
along fixed spatial coordinates.
Formally, this means
\begin{equation}
  \lift{F} :
  \R^{J \times \Gr} \to \R^{I \times \Gr}, \quad
  (x_{j,g})_{j \in J, g \in \Gr}
  \mapsto \big( [F ( (x_{j,g})_{j \in J}) ]_{i} \big)_{i \in I, g \in \Gr}
  \, .
  \label{eq:LiftingDefinition}
\end{equation}
It is not hard to verify $\lift{(F \circ G)} = \lift{F} \circ \lift{G}$ for
$F : \R^J \to \R^I$ and $G : \R^K \to \R^J$.

Given these notations, we can state two final preparatory definitions.
We start by defining the maps that perform the convolutional steps in a CNN.

\begin{definition}\label{def:filtering}
Let $\Gr$ be a finite group
and ${B : \R^{\FirstN{C_1} \times \Gr} \to \R^{\FirstN{k} \times \FirstN{C_1} \times \Gr}}$,
where $k, C_1 \in \N$.
We say that $B$ is \emph{filtering, with $k$ filters}, if there are $a_1, \dots, a_k \in \R^{\Gr}$
such that
\begin{equation}
  B \big( (x_i)_{i \in \FirstN{C_1}} \big)
  = (x_j \ast a_r)_{(r, j) \in \FirstN{k} \times \FirstN{C_1}}
  \qquad \forall \, (x_i)_{i \in \FirstN{C_1}} \in \R^{\FirstN{C_1} \times \Gr} \, .
  \label{eq:FilteringMapDefinition}
\end{equation}
In this case, we write $B \in \filter{k,C_1}$, and set
$\|B\|_{\ell^0_{\mathrm{filter}}} := \sum_{r=1}^k \|a_r\|_{\ell^0}$.
This is well-defined, since the filters $a_1,\dots,a_k$ are uniquely determined by $B$.
\end{definition}

The following definition formalises the first two steps in Figure~\ref{fig:CNNExplanation}.

\begin{definition}\label{def:SpatConvSemCon}
Given $k, C_1, C_2 \in \N$, we say that a map
$T: \R^{\FirstN{C_1} \times \Gr} \to \R^{\FirstN{C_2} \times \Gr}$
is a \emph{spatially-convolutional, semi-connected map with $k$ filters},
if $T$ can be written as $T = \lift{A} \circ B$ for $B \in \filter{k, C_1}$ and
an affine-linear map ${A : \R^{\FirstN{k} \times \FirstN{C_1}} \to \R^{\FirstN{C_2}}}$.
In this case, we write $T \in \conv{k, C_1, C_2}$, and define
\begin{align*}
  \|T\|_{\ell^0_{\mathrm{conv}}}
  := \min \Big\{
            \|A\|_{\ell^0} + \|B\|_{\ell^0_{\mathrm{filter}}}
            \,\Big|
            \begin{array}{l}
              A : \R^{\FirstN{k} \times \FirstN{C_1}} \to \R^{\FirstN{C_2}} \text{ affine-linear}\\
              \text{and } B \in \filter{k, C_1} \text{ with } T = \lift{A} \circ B
            \end{array}
            \!\!
          \Big\} .
\end{align*}
\end{definition}
%

\begin{remark}\label{rem:ConvolutionalMapsAsAffineLinearMaps}
  Every ${T \in \conv{k, C_1, C_2}}$ is affine-linear.
  Furthermore, the number of weights of $T$ as an ``ordinary'' affine-linear map
  can be estimated up to a multiplicative constant by
  $\|T\|_{\ell^0_{\mathrm{conv}}}$; in fact, we have
  \begin{equation}
    \|T\|_{\ell^0} \leq |\Gr|^2 \cdot \|T\|_{\ell^0_{\mathrm{conv}}} \, .
    \label{eq:WeightRelationOrdinaryConvolutional}
  \end{equation}
  To see this, choose an affine-linear map
  $A : \R^{\FirstN{k} \times \FirstN{C_1}} \to \R^{\FirstN{C_2}}$ and a filtering map
  $B \in \filter{k, C_1}$ such that $T = \lift{A} \circ B$ and
  $\|T\|_{\ell^0_{\mathrm{conv}}} = \|A\|_{\ell^0} + \|B\|_{\ell^0_{\mathrm{filter}}}$.
  Furthermore, choose $a_1, \dots, a_k \in \R^{\Gr}$ such that $B$ satisfies
  Equation~\eqref{eq:FilteringMapDefinition},
  and let $A_{\mathrm{lin}} : \R^{\FirstN{k} \times \FirstN{C_1}} \to \R^{\FirstN{C_2}}$ be linear
  and $b \in \R^{\FirstN{C_2}}$ such that $A(\cdot) = b + A_{\mathrm{lin}}(\cdot)$.

  Now, define $\lift{b} := (b_j)_{j \in \FirstN{C_2}, g \in \Gr} \in \R^{\FirstN{C_2} \times \Gr}$.
  It is not hard to see $T = \lift{b} + \lift{A_{\mathrm{lin}}} \circ B$, where the map
  $\lift{A_{\mathrm{lin}}} \circ B : \R^{\FirstN{C_1} \times \Gr} \to \R^{\FirstN{C_2} \times \Gr}$
  is linear.
  Furthermore, for arbitrary $i_0 \in \FirstN{C_1}, j \in \FirstN{C_2}$ and $h_0, g \in \Gr$,
  we have
  \begin{align*}
    \left( \left[\lift{A_{\mathrm{lin}}} \circ B\right] \delta_{i_0, h_0}\right)_{j,g}
    &= \Big[
        \Big(
          A_{\mathrm{lin}} \big(
                             \delta_{i_0, i} \cdot (\delta_{h_0} \ast a_r)_g
                           \big)_{r \in \FirstN{k}, i \in \FirstN{C_1}}
        \Big)
      \Big]_j \\
    &= \sum_{r=1}^k
        (A_{\mathrm{lin}})_{j, (r, i_0)} \cdot a_r (h_0^{-1} g) \, ,
  \end{align*}
  where we identified the linear map $A_{\mathrm{lin}}$ with the matrix associated to it
  (by virtue of the standard basis).
  The above identity implies that
  \begin{align*}
    \|T\|_{\ell^0}
    & = \left\|\lift{b}\right\|_{\ell^0} + \left\| \smash{\lift{A_{\mathrm{lin}}}} \circ B\right\|_{\ell^0}\\
    & \leq \sum_{j = 1}^{C_2} \,
             \sum_{g \in \Gr} \,
               \Indicator_{b_j \neq 0}
           + \sum_{j=1}^{C_2} \,
               \sum_{i_0 = 1}^{C_1} \,
                 \sum_{h_0, g \in \Gr}
                   \sum_{r=1}^k
                     \Indicator_{(A_{\mathrm{lin}})_{j, (r,i_0)} \neq 0}
                     \cdot \Indicator_{a_r (h_0^{-1} g) \neq 0} \\
    & \leq |\Gr| \cdot \|b\|_{\ell^0}
           + |\Gr| \cdot \Big( \max_{r=1,\dots,k} \|a_r\|_{\ell^0} \Big)
             \cdot \sum_{j=1}^{C_2}
                     \sum_{i_0 = 1}^{C_1}
                       \sum_{r=1}^k
                         \Indicator_{(A_{\mathrm{lin}})_{j, (r,i_0)} \neq 0} \\
    & \leq |\Gr| \cdot \|b\|_{\ell^0}
           + |\Gr| \cdot \Big( \max_{r=1,\dots,k} \|a_r\|_{\ell^0} \Big)
             \cdot \|A_{\mathrm{lin}}\|_{\ell^0}
      \leq |\Gr|^2 \cdot \|T\|_{\ell^0_{\mathrm{conv}}} \, .
  \end{align*}
  Here, the identity
  \(
    \sum_{h_0, g \in \Gr} \Indicator_{a_r (h_0^{-1} g) \neq 0}
     = \sum_{h_0 \in \Gr} \|a_r\|_{\ell^0} = |\Gr| \cdot \|a_r\|_{\ell^0}
  \)
  follows from the change of variables $h = h_0^{-1} g$.
  Furthermore, we used in the last step that
  $\|b\|_{\ell^0} + \|A_{\mathrm{lin}}\|_{\ell^0}
   = \|A\|_{\ell^0}
   \leq \|T\|_{\ell^0_{\mathrm{conv}}}$,
  and that $\|a_r\|_{\ell^0} \leq |\Gr|$, since $a_r \in \R^{\Gr}$.
\end{remark}

We now define CNNs similarly to FNNs, with the modification that the affine-linear maps
in the definition of the network are required to be spatially-convolutional, semi-connected.

\begin{definition}\label{def:ConvNets}
Let $L \in \N$, let $\Gr$ be a finite group,
let $C_0, C_1, \dots, C_L \in \N$, and let $k_1, \dots, k_L \in \N$.
A \emph{convolutional neural network} $\Phi$ with $L$ layers,
channel counts $(C_0, C_1, \dots, C_L)$, and filter counts $(k_1, \dots, k_L)$
is a tuple $\Phi = (T_1, \dots, T_L)$
where $T_\ell \in \conv{k_\ell,C_{\ell-1},C_{\ell}}$ for $\ell = 1,\dots,L$.

For a convolutional neural network $\Phi = (T_1, \dots, T_L)$ and an activation function
$\varrho:\R\to \R$, we define the \emph{$\varrho$-realisation} of $\Phi$ as
\begin{align*}
  R_\varrho(\Phi) :
  \R^{\FirstN{C_0} \times \Gr} \to \R^{\FirstN{C_L} \times \Gr} \, ,
  \quad
   \,\, x &\mapsto x_L, 
\end{align*}
where $x_0 := x$, $x_{\ell + 1} := \varrho \big( T_{\ell + 1} (x_\ell) \big)$
for $0 \leq \ell \leq L - 2$, and $x_L := T_L(x_{L-1})$.
Here, we again apply $\varrho$ component-wise.

The \emph{number of channels of $\Phi$} is
$C(\Phi) : = \sum_{\ell=0}^L C_\ell$.
The number ${C_0 = C_0 (\Phi) \in \N}$ is called the \emph{number of input channels of $\Phi$},
while ${C_L = C_L (\Phi) \in \N}$ is the \emph{number of output channels}.
Finally,
$W_{\mathrm{conv}}(\Phi) := \sum_{\ell=1}^L \|T_\ell\|_{\ell^0_{\mathrm{conv}}}$
is \emph{the number of weights}.
\end{definition}

\begin{rem*}
  It could be more natural to call $W_{\mathrm{conv}}(\Phi)$ the \emph{number of free parameters},
  instead of ``the number of weights''.
  We chose the present terminology primarily to be consistent
  with the established terminology for FNNs.
\end{rem*}

\begin{remark}\label{rem:ConvolutionalNetsAsFullyConnectedNets}
  With the identification $\R^{\FirstN{C_{\ell}} \times \Gr} \cong \R^{C_{\ell} \cdot |\Gr|}$,
  each CNN $\Phi = (T_1, \dots, T_L)$ is also an FNN, simply because
  each of the maps $T_\ell \in \conv{k_\ell, C_{\ell-1}, C_\ell}$ is an affine-linear map
  $T_\ell : \R^{\FirstN{C_{\ell-1}} \times \Gr} \to \R^{\FirstN{C_\ell} \times \Gr}$.
  When interpreting $\Phi$ as an FNN, it has architecture
  $\arch{\Phi} = (C_0 \cdot |\Gr|, C_1 \cdot |\Gr|, \dots, |C_L| \cdot |\Gr|)$,
  and thus $N (\Phi) \leq |\Gr| \cdot C(\Phi)$.
  Furthermore, as a consequence of Remark~\ref{rem:ConvolutionalMapsAsAffineLinearMaps}, we see
  \begin{equation}
    W(\Phi)
    = \sum_{\ell=1}^L \|T_\ell\|_{\ell^0}
    \leq |\Gr|^2 \cdot \sum_{\ell=1}^L \|T_\ell\|_{\ell^0_{\mathrm{conv}}}
    = |\Gr|^2 \cdot W_{\mathrm{conv}} (\Phi) \, .
    \label{eq:ConvNetAsFullyConnectedWeightControl}
  \end{equation}
\end{remark}


As just seen, CNNs are special FNNs.
Hence, it is natural to ask to what extent these networks can achieve
the same approximation properties as FNNs.
It turns out that the restriction to CNNs is significant,
since CNNs can only approximate so-called translation equivariant functions.
To make the concept of translation equivariance more precise, and, in particular,
meaningful for functions with different input and output dimensions,
we first introduce the notion of vectorisation:
For a set $I$ and a function $H : \R^{\Gr} \to \R^{\Gr}$ we define the
\emph{$I$-vectorisation} $\vect{H}{I}$ of $H$ as
\begin{equation}
  \vect{H}{I} :
  \R^{I \times \Gr} \to \R^{I \times \Gr}, \quad
  (x_{i})_{i \in I} \mapsto \big( H(x_i) \big)_{i \in I} \, .
  \label{eq:VectorisationDefinition}
\end{equation}

Next, we define the concept of translation equivariance.

\begin{definition}\label{def:TranslationCoVariant}
  Let $\Gr$ be a finite group, $I,J$ index sets, and
  ${F : \R^{I \times \Gr} \to \R^{J \times \Gr}}$.
  We say that $F$ is \emph{translation equivariant}, if
  ${F \circ \vect{S_g}{I} = \vect{S_g}{J} \circ F}$ for all ${g \in \Gr}$,
  where $\vect{S_g}{I}$ and $\vect{S_g}{J}$ are, respectively, the $I$-vectorisation
  and the $J$-vectorisation of the \emph{shift operator}
  \[
    S_g : \R^{\Gr} \to \R^{\Gr}, (x_h)_{h \in \Gr} \mapsto (x_{g^{-1} h})_{h \in \Gr}.
  \]
\end{definition}


As previously announced, every realisation of a CNN is translation equivariant,
as the following proposition demonstrates.

\begin{proposition}\label{prop:ConvNetsAreTranslationInvariant}
Let $\Gr$ be a finite group,
let $\varrho : \R \to \R$ be any function, and let $\Phi$ be a CNN.
Then the $\varrho$-realisation $R_\varrho (\Phi)$ is translation equivariant.
\end{proposition}

\begin{proof}
  Directly from the definition of the convolution in Equation~\eqref{eq:ConvolutionDefinition},
  we see that convolutions are translation equivariant;
  that is, $S_g (x \ast a) = (S_g \, x) \ast a$ for $a,x \in \R^{\Gr}$ and $g \in \Gr$.
  Thus,
  \(
    \vect{S_g}{\FirstN{k} \times \FirstN{C_1}} \circ B
    = B \circ \vect{S_g}{\FirstN{C_1}}
  \)
  for all $g \in \Gr$ and any
  filtering map $B$ as in Equation~\eqref{eq:FilteringMapDefinition}.

  Now, for a permutation $\pi$ of $\Gr$, let us write
  $C_\pi : \R^{\Gr} \to \R^{\Gr},
   (x_g)_{g \in \Gr} \mapsto (x_{\pi(g)})_{g \in \Gr}$.
  A direct computation shows for any $F : \R^I \to \R^J$ that
  ${\lift{F} \circ \vect{C_\pi}{I} = \vect{C_\pi}{J} \circ \lift{F}}$.
  Clearly, $S_g = C_\pi$ for a suitable permutation $\pi = \pi_g$.
  Overall, we thus see for any $T = \lift{A} \circ B \in \conv{k,C_1,C_2}$ that
  \[
    T \circ \vect{S_g}{\FirstN{C_1}}
    = \lift{A} \circ \vect{S_g}{\FirstN{k} \times \FirstN{C_1}} \circ B
    = \vect{S_g}{\FirstN{C_2}} \circ \lift{A} \circ B
    = \vect{S_g}{\FirstN{C_2}} \circ T
    \quad \forall \, g \in \Gr \, .
  \]
  Since the activation function $\varrho$ is applied component-wise,
  it immediately follows that
  \(
    \varrho \big( T ( \vect{S_g}{\FirstN{C_1}} x ) \big)
    = \vect{S_g}{\FirstN{C_2}} \big( \varrho (T x) \big)
  \)
  for every $T \in \conv{k,C_1,C_2}$
  and all $x \in \R^{\FirstN{C_1} \times \Gr}$.
  By iterating this observation, we get the claim.
\end{proof}

The proposition shows that all realisations of CNNs are translation equivariant.
We note that the approximation theory of CNNs has been studied before,
for instance in the works \cite{cohen2016expressive, UniConv}.
The CNNs considered in these works, however, are different from the definition used in the present paper.
For the CNNs studied in the present paper, their universality for the class of continuous
translation equivariant functions has been established in \cite[Theorem 3.1]{YarotskyWasFirstAgain}.
Nevertheless, until now it was not known what kind of
\emph{approximation rates} these CNNs yield.
In the next section, we show that there is in fact a fundamental connection
between the approximation capabilities of FNNs and these CNNs.

\section{\texorpdfstring{Approximation rates of convolutional neural networks\\
for translation equivariant functions}
{Approximation rates of convolutional neural networks
for translation equivariant functions}}
\label{sec:ApproximationUsingCNN}

{We start by demonstrating in Subsection \ref{subsec:Transference} how one can associate to each FNN a CNN,
such that the first coordinate of the realisation of the CNN coincides with the realisation of the FNN.
In addition, in Remark~\ref{rem:TheRemarkForTheConverse} we show the converse statement,
i.e., that each CNN can be transformed into an associated FNN.

Afterwards, we demonstrate in Section~\ref{sec:EquivalenceOfRates} how this yields
an equivalence between the approximation rates of CNNs and FNNs.
We close with a concrete example showing how our results can be used to translate
approximation results for FNNs into approximation results for CNNs.}

\subsection{The transference principle}
\label{subsec:Transference}

We will measure approximation rates with respect to $L^p$ norms of vector-valued functions.
For these (quasi)-norms, we use the following convention:
For any $d \in \N$, any finite index set $J$, any measurable subset $\Omega \subset \R^{d}$,
any $p \in (0,\infty]$, and any measurable function $f: \Omega \to \R^{J}$, we define
\[
  \|f \|_{L^{p}(\Omega, \R^{J})}
  := \big\| x \mapsto \| f(x) \|_{\ell^p} \big\|_{L^{p}(\Omega)} \, .
\]
Note that this implies for $F: \Omega \to \R^{J\times \Gr}$ and $p < \infty$ that
\begin{equation}
  \|F\|_{L^{p}(\Omega, \R^{J\times \Gr})}^p
  = \sum_{g \in \Gr} \| F_{g} \|_{L^{p}(\Omega, \R^{J})}^p \, ,
  \label{eq:VectorValuedNormAsSum}
\end{equation}
where $F_g := (F)_g := \pi_g^J \circ F : \Omega \to \R^{J}$
denotes the \emph{$g$-th component of $F$}.
Here, the function $\pi_g^J$ is the \emph{projection onto the $g$-th component}, given by
\begin{equation}
  \pi_g^J : \R^{J \times \Gr}          \to     \R^J,
            (x_{j,h})_{j \in J, h \in \Gr} \mapsto (x_{j,g})_{j \in J}
  \, .
  \label{eq:CoordinateProjectionDefinition}
\end{equation}

\begin{rem*}
One could also define $\|F\|_{L^{p}(\Omega, \R^{J})}$ as
$\| \, |F| \, \|_{L^p (\Omega)}$, where $|F| (x) = |F(x)|$ denotes the Euclidean norm of $F(x)$.
It is not hard to see that both (quasi)-norms are equivalent since $J$ and $\Gr$ are finite;
furthermore, the constant of the norm equivalence only depends on $|J|$, $|\Gr|$, and $p$.
\end{rem*}

We denote the identity element of $\Gr$ by $1$ and observe that if
$F,G: \R^{I \times \Gr} \to \R^{J \times \Gr}$ are translation equivariant and
$(F)_1 = (G)_1$ then $F = G$; indeed, it suffices to show for all $g \in \Gr$
that $(F)_g = (G)_g$. This holds since we have
\begin{equation}
  (F)_g
  = \pi_1^J \circ \vect{S_{g^{-1}}}{J} \circ F
  = \pi_1^J \circ F \circ \vect{S_{g^{-1}}}{I}
  = (F)_1 \circ \vect{S_{g^{-1}}}{I}
  \label{eq:CovariantComponentTrick}
\end{equation}
for every translation equivariant function
$F : \R^{I \times \Gr} \to \R^{J \times \Gr}$.

Given a finite index set $I \neq \emptyset$, we say that a subset
$\Omega \subset \R^{I \times \Gr}$ is \emph{$\Gr$-invariant},
if $\vect{S_g}{I} (\Omega) \subset \Omega$ for all $g \in \Gr$.
An example of such a set is $\prod_{i \in I} \Omega_i^{\Gr}$,
where the sets $\Omega_i \subset \R$ for $i \in I$ can be chosen arbitrarily.
Since $\vect{S_g}{I} x = P_g x$ for all $x \in \R^{I \times \Gr}$ and a suitable
permutation matrix $P_g$, Equations~\eqref{eq:VectorValuedNormAsSum} and
\eqref{eq:CovariantComponentTrick} show for any $p \in (0,\infty)$, any measurable
$\Gr$-invariant set $\Omega \subset \R^{I \times \Gr}$ and any two
translation equivariant functions $F,G : \R^{I \times \Gr} \to \R^{J \times \Gr}$ that
\begin{equation}
  \begin{split}
    \| F - G \|_{L^p(\Omega, \R^{J \times \Gr})}
    & = \Big(
          \sum_{g\in \Gr}
            \| (F)_g - (G)_g \|_{L^p(\Omega,  \R^J)}^p
        \Big)^{1/p} \\
    & = |\Gr|^{{1}/{p}} \cdot \| (F)_1 - (G)_1 \|_{L^p(\Omega, \R^J)} \, ,
  \end{split}
  \label{eq:TranslationCovariantNorm}
\end{equation}
and this clearly remains true for $p = \infty$.


We can now state the transference principle between FNNs and CNNs.

\begin{theorem}\label{thm:main}
  Let $\Gr$ be a finite group, let $\eps \in [0,\infty)$,
  $p\in (0, \infty]$, and $C_0,N \in \N$.
  Let $\Omega \subset \R^{\FirstN{C_0} \times \Gr}$ be $\Gr$-invariant and measurable,
  and let $\varrho : \R \to \R$ be measurable.

  Let $F: \R^{\FirstN{C_0} \times \Gr} \! \to \R^{\FirstN{N} \times \Gr}$
  be measurable and translation equivariant,
  and let $\Phi$ be an FNN of architecture
  $\arch{\Phi} = (C_0 \cdot |\Gr|, N_1, \dots, N_{L-1}, N)$ satisfying
  ${\|(F)_1 - R_\varrho(\Phi)\|_{L^p(\Omega, \R^{N})} \leq \eps}$.

  Then there is a CNN $\Psi$ with channel counts
  $(C_0, N_1, \dots, N_{L-1}, N)$, with filter counts $(N_1 \cdot C_0, 1, \dots, 1)$,
  and such that
  \(
    \| F - R_\varrho (\Psi) \|_{L^p(\Omega,\R^{\FirstN{N} \times \Gr})}
    \leq |\Gr|^{1/p} \cdot \eps
  \)
  and $W_{\mathrm{conv}} (\Psi) \leq 2 \cdot W(\Phi)$.
  Here, we use the convention $|\Gr|^{1/\infty} = 1$.
\end{theorem}

\begin{rem*}
  {1) The proof shows that the network $\Psi$ can be chosen independently
   of the activation function $\varrho$, unless $\Phi = (V_1, \dots, V_L)$ with $V_\ell = 0$
   for some $\ell \in \{1,\dots,L\}$.}


  {2)}
  Since we can choose $\eps = 0$ and $\Omega = \R^{\FirstN{C_0} \times \Gr}$,
  the theorem shows in particular that if
  ${(F)_1 = R_\varrho (\Phi)}$ for an FNN $\Phi$ of
  architecture $(C_0 \cdot |\Gr|, N_1,\dots,N_{L-1},N)$, then $F = R_\varrho (\Psi)$
  for a CNN $\Psi$ with channel counts
  $(C_0, N_1,\dots,N_{L-1},N)$, with filter counts $(N_1 \cdot C_0, 1, \dots, 1)$,
  and with $W_{\mathrm{conv}} (\Psi) \leq 2\cdot W(\Phi)$.


  {3) In addition to the number of layers and weights, the \emph{complexity of the individual
   weights} can also be relevant. Given a set $\Lambda \subset \R$, we say that an affine-linear
   map $V : \R^I \to \R^J$ \emph{has weights in $\Lambda$} if $V(x) = b + A \, x$
   for all $x \in \R^I$ and certain $b \in \Lambda^J$ and $A \in \Lambda^{J \times I}$.
   Likewise, we say that an FNN $\Phi = (V_1, \dots, V_L)$ has weights in $\Lambda$ if all
   $V_\ell$ have weights in $\Lambda$.

   The proof of the theorem shows that if the FNN $\Phi = (V_1,\dots,V_L)$
   has weights in $\Lambda$, and if $V_\ell \neq 0$ for some $\ell$,
   then the CNN $\Psi = (T_1,\dots,T_L)$ constructed in the theorem satisfies
   $T_\ell = \lift{A_\ell} \circ B_\ell$ where all $A_\ell, B_\ell$ have weights in
   $\Lambda \cup \{0,1\}$.}
\end{rem*}

\begin{proof}
In view of Equation~\eqref{eq:TranslationCovariantNorm} and since $R_\varrho (\Psi)$ is translation
equivariant (see Proposition~\ref{prop:ConvNetsAreTranslationInvariant}), we need only show that
there is a CNN $\Psi$ with the asserted channel counts, filter counts,
and number of weights, and such that $(R_\varrho (\Psi))_{1} = R_\varrho (\Phi)$.

Let $\Phi = (V_1,\dots,V_L)$. For brevity, set $N_0 := C_0$ and $N_L := N$, and furthermore
$k_1 := N_1 \cdot C_0$ and $k_\ell := 1$ for $\ell \in \{2,\dots,L\}$.

We first handle a few special cases, in order to avoid tedious case distinctions later on.
First, if $\|V_L\|_{\ell^0} = 0$, then $R_\varrho (\Phi) \equiv 0$.
Then, let $\Psi = (T_1,\dots,T_L)$, where $T_\ell = \lift{A_\ell} \circ B_\ell$ with
$A_\ell : \R^{\FirstN{k_\ell} \times \FirstN{N_{\ell - 1}}} \to \R^{\FirstN{N_\ell}}, x \mapsto 0$,
and with
\[
  B_\ell :
  \R^{\FirstN{N_{\ell - 1}} \times \Gr}
  \to \R^{\FirstN{k_\ell} \times \FirstN{N_{\ell - 1}} \times \Gr},
  (x_i)_{i \in \FirstN{N_{\ell - 1}}}
  \mapsto (x_i \ast 0)_{r \in \FirstN{k_\ell}, i \in \FirstN{N_{\ell - 1}}}
\]
for all $\ell \in \{1,\dots,L\}$.
It is then trivial to verify that $\Psi$ has the desired number of filters and channels,
that $(R_\varrho (\Psi))_1 \equiv 0$, and that $\|T_\ell\|_{\ell^0_{\mathrm{conv}}} = 0$
for all $\ell = 1,\dots,L$, so that $W_{\mathrm{conv}} (\Psi) = 0 \leq 2 \cdot W(\Phi)$.

Next, if $\|V_L\|_{\ell^0} > 0$, but $\|V_\ell\|_{\ell^0} = 0$ for some $\ell \in \{1,\dots,L-1\}$,
then there is some $c \in \R^N$ such that $\|c\|_{\ell_0} \leq W(\Phi)$ and
$R_\varrho (\Phi) \equiv c$.
Indeed, $c = V_L c_0$ for some $c_0 \in \R^{N_{L-1}}$.
Besides, for any $A \in \R^{n \times k}$, $b \in \R^n$ and $x \in \R^k$, we have
$(A \, x + b)_j = b_j + \sum_{\ell=1}^k A_{j,\ell} x_\ell$,
which shows that
$\Indicator_{(A \, x + b)_j \neq 0}
 \leq \Indicator_{b_j \neq 0} + \sum_{\ell=1}^k \Indicator_{A_{j,\ell} \neq 0}$,
and hence
\vspace{-0.15cm}
\[
  \|A \, x + b\|_{\ell^0}
  = \sum_{j=1}^n \Indicator_{(A \, x + b)_j \neq 0}
  \leq \sum_{j=1}^n \Indicator_{b_j \neq 0}
       + \sum_{j=1}^n \,
           \sum_{\ell=1}^k
             \Indicator_{A_{j, \ell} \neq 0}
  =    \|A (\cdot) + b\|_{\ell^0} \, .
\]
Therefore, $\|c\|_{\ell^0} = \|V_L c_0\|_{\ell^0} \leq \|V_L\|_{\ell^0} \leq W(\Phi)$.

Given such a vector $c \in \R^N$ with $R_\varrho (\Phi) \equiv c$ and $\|c\|_{\ell^0} \leq W(\Phi)$,
we define $\Psi = (T_1,\dots,T_L)$, where $T_\ell = \lift{A_\ell} \circ B_\ell$,
and where $B_1,\dots,B_{L}$ and $A_1,\dots,A_{L-1}$ are defined as in the previous case,
and where
\(
  A_L :
  \R^{\FirstN{k_L} \times \FirstN{N_{L-1}}} \to \R^{\FirstN{N_L}},
  x \mapsto c .
\)
This is well-defined, since $N_L = N$, whence $c \in \R^N \cong \R^{\FirstN{N_L}}$.
It is not hard to see that $(R_\varrho (\Psi))_1 \equiv c \equiv R_\varrho (\Phi)$,
that $\Psi$ has the right number of filters and channels, and that
$W_{\mathrm{conv}} (\Psi) = \|A_L\|_{\ell^0} = \|c\|_{\ell^0} \leq W(\Phi)$.

\medskip{}

In the following, we can thus assume $\|V_\ell\|_{\ell^0} > 0$ for all $\ell \in \{1,\dots,L\}$.
Below, we will repeatedly make use of the following fact:
If $v \in \R^{\Gr}$, and if we define $v^\ast \in \R^{\Gr}$ by $v_g^\ast := v_{g^{-1}}$ for $g \in \Gr$, then
\begin{equation}
  (x \ast v^\ast)_1
  = \sum_{h \in \Gr} x_h \, v^\ast_{h^{-1}}
  = \sum_{h \in \Gr} x_h \, v_h
  = \langle x, v \rangle_{\R^{\Gr}}
  \qquad \forall \, x \in \R^{\Gr} \, .
  \label{eq:ScalarProductAsConvolution}
\end{equation}
Furthermore, $x \ast \delta_1 = x$ for all $x \in \R^{\Gr}$,
where $(\delta_1)_1 = 1$ and $(\delta_1)_g = 0$ for $g \in \Gr \setminus \{1\}$.
We remark that this way of expressing an inner product as a convolution
has also been used in \cite{UniConv} to analyse the expressivity of CNNs.

Recall that $\Phi = (V_1, \dots, V_L)$.
Since $V_1 : \R^{\FirstN{C_0} \times \Gr} \to \R^{\FirstN{N_1}}$ is affine-linear,
there are $v_j^i \in \R^{\Gr}$ and $b_j \in \R$
(for $j \in \FirstN{N_1}$ and $i \in \FirstN{C_0}$) such that
$V_1(\cdot) = b + V_1^{\mathrm{lin}}(\cdot)$, where $b = (b_j)_{j \in \FirstN{N_1}}$ and
\[
  V_1^{\mathrm{lin}} :
  \R^{\FirstN{C_0} \times \Gr} \to \R^{\FirstN{N_1}}, \quad
  (x_i)_{i \in \FirstN{C_0}} \mapsto \Big(
                                       \sum_{i=1}^{C_0}
                                         \langle x_i, v_j^i \rangle_{\R^{\Gr}}
                                     \Big)_{j \in \FirstN{N_1}} \, .
\]

We now define $T_1 := \lift{A_1} \circ B_1 \in \conv{C_0 \cdot N_1,C_0,N_1}$, where
\[
  B_1 : \R^{\FirstN{C_0} \times \Gr} \to \R^{\FirstN{N_1}
                                                \times \FirstN{C_0}
                                                \times \FirstN{C_0}
                                                \times \Gr}, \quad
  (x_i)_{i \in \FirstN{C_0}} \mapsto \big(
                                       x_{\iota} \ast (v^i_j)^\ast
                                     \big)_{j \in \FirstN{N_1},
                                               i,\iota \in \FirstN{C_0}}
\]
and $A_1 : \R^{\FirstN{N_1} \times \FirstN{C_0} \times \FirstN{C_0}} \to \R^{\FirstN{N_1}},
           y \mapsto b + A_1^{\mathrm{lin}} \, y$, where
\[
  A_1^{\mathrm{lin}}
  : \R^{\FirstN{N_1} \times \FirstN{C_0} \times \FirstN{C_0}} \to \R^{\FirstN{N_1}} \, , \quad
  (y_{j,i,\iota})_{j \in \FirstN{N_1}, i,\iota \in \FirstN{C_0}}
  \mapsto \Big(
            \sum_{i = 1}^{C_0}
              \Indicator_{v^i_{\ell} \neq 0} \cdot y_{\ell,i,i}
          \Big)_{\ell \in \FirstN{N_1}}
  \, .
\]
As a consequence of these definitions, we see for
$x = (x_{j,g})_{j \in \FirstN{C_0}, g \in \Gr} \in \R^{\FirstN{C_0} \times \Gr}$ and
$\ell \in \FirstN{N_1}$ that
\begin{align*}
  \big( (\lift{A_1} \circ B_1) x \big)_{\ell, 1}
  & = \Big(
        A_1 \big[
              (B_1 x)_{j, i, \iota, 1}
            \big]_{(j,i,\iota) \in \FirstN{N_1} \times \FirstN{C_0} \times \FirstN{C_0}}
      \Big)_\ell \\
   & = \Big(
        A_1 \big[
              (x_{\iota} \ast (v^i_j)^\ast)_1
            \big]_{(j,i,\iota) \in \FirstN{N_1} \times \FirstN{C_0} \times \FirstN{C_0}}
      \Big)_\ell \\
  ({\scriptstyle{\text{Def. of } A_1 \text{ and Eq. } \eqref{eq:ScalarProductAsConvolution}}})
  & = b_\ell
      + \sum_{i=1}^{C_0}
          \big[
            \langle x_i, v^i_\ell \rangle_{\R^{\Gr}}
            \cdot \Indicator_{v^i_\ell \neq 0}
          \big] 
     = b_\ell
      + \sum_{i=1}^{C_0}
          \langle x_i, v^i_\ell \rangle_{\R^{\Gr}}
    = (V_1 \, x)_\ell \, .
\end{align*}
In other words, with the projection map $\pi_1^{\FirstN{N_1}}$
defined in Equation~\eqref{eq:CoordinateProjectionDefinition}, we have
\begin{equation}
  \pi_1^{\FirstN{N_1}} \circ T_1
  = \pi_1^{\FirstN{N_1}} \circ \lift{A_1} \circ B_1
  = V_1 \, .
  \label{eq:FirstLayerProperty}
\end{equation}
Furthermore, we see directly from the definition of $A_1^{\mathrm{lin}}$ that
\[
  \big( A_1^{\mathrm{lin}} \, \delta_{j, i, \iota} \big)_\ell
   = \sum_{n=1}^{C_0}
      \big( \Indicator_{v^n_\ell \neq 0} \cdot \delta_{j, i, \iota} \big)_{\ell,n,n} 
   = \delta_{i, \iota} \cdot \Indicator_{v^{i}_{j} \neq 0} \cdot \delta_{j, \ell}
\]
for all $j, \ell \in \FirstN{N_1}$ and $i, \iota \in \FirstN{C_0}$.
Therefore,
\[
  \|A_1^{\mathrm{lin}}\|_{\ell^0}
  = \sum_{j=1}^{N_1}
      \sum_{\ell=1}^{N_1}
        \sum_{i,\iota \in \FirstN{C_0}} \!\!\!
          \Indicator_{( A_1^{\mathrm{lin}} \delta_{j,i,\iota} )_\ell \neq 0} 
   =\! \sum_{j=1}^{N_1} \,
         \sum_{i=1}^{C_0}
           \Indicator_{v^i_j \neq 0}
  \leq \sum_{j=1}^{N_1}
         \sum_{i=1}^{C_0}
           \|v^i_j\|_{\ell^0} \!
  = \|V_1^{\mathrm{lin}}\|_{\ell^0} ,
\]
and hence $\|A_1\|_{\ell^0}
           =    \|b\|_{\ell^0} + \|A_1^{\mathrm{lin}}\|_{\ell^0}
           \leq \|b\|_{\ell^0} + \|V_1^{\mathrm{lin}}\|_{\ell^0}
           =    \|V_1\|_{\ell^0}$.
Next, note
\[
  \|B_1\|_{\ell^0_{\mathrm{filter}}}
  = \sum_{j=1}^{N_1} \,
      \sum_{i=1}^{C_0}
        \|(v^i_j)^\ast\|_{\ell^0}
  = \|V_1^{\mathrm{lin}}\|_{\ell^0}
  \leq \|V_1\|_{\ell^0} \, ,
\]
which finally implies
\(
  \|T_1\|_{\ell^0_{\mathrm{conv}}}
  \leq \|A_1\|_{\ell^0} + \|B_1\|_{\ell^0_{\mathrm{filter}}}
  \leq 2 \cdot \|V_1\|_{\ell^0} \, .
\)

\medskip{}

Next, for $\ell \in \{2,\dots,L\}$, we define
$T_\ell := \lift{V_\ell} \circ B_\ell \in \conv{1,N_{\ell-1},N_\ell}$, where
\[
  B_\ell : \R^{\FirstN{N_{\ell - 1}} \times \Gr}
           \to \R^{\FirstN{N_{\ell - 1}} \times \Gr}, \quad
  x = (x_i)_{i \in \FirstN{N_{\ell - 1}}}
  \mapsto \big(
            x_i \ast \delta_1
          \big)_{i \in \FirstN{N_{\ell - 1}}} = x \, .
\]
Note because of $B_\ell \, x = x$ that $T_\ell = \lift{V_\ell}$.
Furthermore, since we excluded the case $\|V_\ell\|_{\ell^0} = 0$ at the beginning of the proof,
we have
$\|B_\ell\|_{\ell^0_{\mathrm{filter}}} = \|\delta_1\|_{\ell^0} = 1 \leq \|V_\ell\|_{\ell^0}$.
Therefore,
\[
  \|T_\ell\|_{\ell^0_{\mathrm{conv}}}
  \leq \|V_\ell\|_{\ell^0} + \|B_\ell\|_{\ell^0_{\mathrm{filter}}}
  \leq 2 \cdot \|V_\ell\|_{\ell^0} \, .
\]

We now define the CNN $\Psi := (T_1,\dots,T_L)$, noting that this network indeed
has the required number of channels and filters, and that
\[
  W_{\mathrm{conv}}(\Psi)
  = \sum_{\ell=1}^L \|T_\ell\|_{\ell^0_{\mathrm{conv}}}
  \leq 2 \sum_{\ell=1}^L \|V_\ell\|_{\ell^0}
  = 2 \cdot W(\Phi).
\]
By Proposition~\ref{prop:ConvNetsAreTranslationInvariant},
$R_\varrho (\Psi)$ is translation equivariant.
Since $F$ is also translation equivariant,
Equation~\eqref{eq:TranslationCovariantNorm} shows that we only need to verify
$[R_\varrho(\Psi)]_1 = R_\varrho(\Phi)$.
But this is easy to see: We saw above that $T_\ell = \lift{V_\ell}$ for
$\ell \in \{2,\dots,L\}$, which easily implies
$\pi_1^{\FirstN{N_\ell}} \circ T_\ell
 = V_\ell \circ \pi_1^{\FirstN{N_{\ell - 1}}}$.
Since the activation function $\varrho$
is applied component-wise, we conclude as desired that
\begin{align*}
  [R_\varrho (\Psi)]_1
  = \pi_1^{\FirstN{N_L}} \circ R_\varrho (\Psi)
  & = (V_L \circ \varrho \circ V_{L-1} \circ \cdots \circ \varrho \circ V_2)
      \circ \pi_1^{\FirstN{N_1}} \circ \varrho \circ T_1 \\
  ({\scriptstyle \text{Eq.~} \eqref{eq:FirstLayerProperty}
                 \text{ and } \pi_1^{\FirstN{N_1}} \circ \varrho
                              = \varrho \circ \pi_1^{\FirstN{N_1}}})
  & = (V_L \circ \varrho \circ V_{L-1} \circ \cdots \circ \varrho \circ V_2)
      \circ \varrho \circ V_1 \\
  & = R_\varrho (V_1,\dots,V_L) = R_\varrho (\Phi) \, .
  \qedhere
\end{align*}
\end{proof}

The following remark contains the (easier) converse of Theorem~\ref{thm:main},
showing that to each CNN $\Psi$ one can construct an associated FNN $\Phi$
such that ${R_\varrho (\Phi) = [ R_\varrho (\Psi) ]_1}$.

\begin{remark}\label{rem:TheRemarkForTheConverse}
Let $\Psi = (T_1, \dots, T_L)$ be a CNN with channel counts $(C_0, C_1, \dots, C_L)$
and filter counts $(k_1, \dots, k_L)$.
The \emph{FNN associated to $\Psi$}
is $\Phi^\Psi := (T_1, \dots, T_{L-1}, T_{L}^1)$,
where $T_L^1 := \pi_{1}^{\FirstN{C_L}} \circ T_L$.

\smallskip{}

The properties of the network $\Phi^\Psi$ are closely related to those of $\Psi$;
in particular, the following holds:
\begin{itemize}
  \item $R_\varrho (\Phi^\Psi) = [R_\varrho (\Psi)]_1$;
        \vspace{0.1cm}

  \item if $\| F - R_\varrho (\Psi) \|_{L^p(\Omega,\R^{\FirstN{C_L} \times \Gr})}
            \leq \eps \vphantom{\sum_j}$
        for some $p \in (0,\infty]$ and some measurable function
        $F: \R^{\FirstN{C_0} \times \Gr} \to \R^{\FirstN{C_L} \times \Gr}$, then
        ${\|(F)_1 - R_\varrho(\Phi^\Psi)\|_{L^p(\Omega, \R^{C_L})} \leq \eps}$;
        \vspace{0.1cm}

  \item $\Phi^\Psi$ has architecture
        $\arch{\Phi^\Psi} = (|\Gr| \cdot C_0, |\Gr| \cdot C_1, \dots, |\Gr|\cdot C_{L-1}, C_L)$,
        and hence $N(\Phi^{\Psi}) \leq |\Gr| \cdot C(\Psi)$; and
        \vspace{0.1cm}

  \item $W(\Phi^\Psi) \leq |\Gr|^2 \cdot W_{\mathrm{conv}}(\Psi)$.
\end{itemize}
The very last property is a consequence of Equation~\eqref{eq:WeightRelationOrdinaryConvolutional},
combined with the estimate $\|\pi_1^{\FirstN{C_L}} \circ T_L\|_{\ell^0} \leq \|T_L\|_{\ell^0}$.
\end{remark}

\subsection{Equivalence of approximation rates}\label{sec:EquivalenceOfRates}

For $C_0, N\in \N$ and a given function class
$\mathcal{C} \subset \{ F :\R^{\FirstN{C_0} \times \Gr} \to \R^{\FirstN{N}}\}$, we call
\[
  \equi{\mathcal{C}}
  := \left\{
      G :\R^{\FirstN{C_0} \times \Gr} \to \R^{\FirstN{N} \times \Gr}
      \,\middle|\,
      G \text{ translation equivariant and }
      (G)_1 \in \mathcal{C}
    \right\}
\]
the \emph{equivariant function class associated to $\mathcal{C}$}.

In combination, Theorem~\ref{thm:main} and Remark~\ref{rem:TheRemarkForTheConverse} imply that for
any function class $\mathcal{C}\subset \{ F :\R^{\FirstN{C_0} \times \Gr} \to \R^{\FirstN{N}}\}$
consisting of measurable functions,
the approximation rate of FNNs in terms of the number of neurons (or number of weights) is
equivalent (up to multiplicative constants that depend only on $|\Gr|$ and $p$)
to the approximation rate of CNNs in terms of the number of channels
(or number of weights) for the associated equivariant function class $\equi{\mathcal{C}}$.

\medskip

As a result, all upper and lower approximation bounds established
for FNNs (such as for instance
\cite{Barron1993,
      bolcskei2017optimal,
      Cybenko1989,
      Maiorov1999LowerBounds,
      Mhaskar:1996:NNO:1362203.1362213,
      Mhaskar1993,
      PetV2018OptApproxReLU,
      YAROTSKY2017103})
directly imply the same bounds for CNNs
for the corresponding translation equivariant function classes.
As a concrete example of this, we now state the approximation theorem for CNNs
that corresponds to \cite[Theorem 3.1]{PetV2018OptApproxReLU}.

\begin{proposition}\label{prop:OurStuffCNNVersion}
  Let $\Gr$ be a finite group, let $C_0 \in \N$ and $\beta, p \in (0,\infty)$.
  There exists a ${c = c(C_0,|\Gr|,\beta,p) > 0}$ such that for every
  $\eps \in (0, 1/2)$ and every translation equivariant function
  ${F : \R^{\FirstN{C_0} \times \Gr} \to \R^{\Gr}}$ such that
  ${\|(F)_1\|_{C^\beta ([-1/2, 1/2]^{\FirstN{C_0} \times \Gr})} \leq 1}$,
  there is a CNN $\Psi_\eps^F$ with at most
  ${(2 + \lceil \log_2 \beta \rceil_+) \cdot \big( 11 + \beta/(C_0 |\Gr|) \big)}$
  layers and such that
  \[
    W_{\mathrm{conv}} (\Psi_\eps^F) \leq c \cdot \eps^{-C_0 |\Gr|/\beta}
    \qquad \text{and} \qquad
    \|
      R_\varrho (\Psi_\eps^F)
      - F
    \|_{L^p\left(\left[-\frac{1}{2}, \frac{1}{2}\right]^{\FirstN{C_0} \times \Gr}; \R^{\Gr}\right)}
    \leq \eps \, .
  \]
  Here, $\varrho : \R \to \R, x \mapsto x_+ := \max \{0,x\}$ is the ReLU.
\end{proposition}

\begin{rem*}
  1) For the precise definition of $\|f\|_{C^{\beta}}$, we refer to
  \cite[Section 3.1]{PetV2018OptApproxReLU}.

  \smallskip{}

  2) Note that the exponent $C_0 |\Gr| / \beta$ is of the form ``$\text{input dimension} / \text{smoothness}$'',
  as usual for such approximation results.

  \smallskip{}

  3) Under a certain \emph{encodability assumption} (see \cite[Section 4]{PetV2018OptApproxReLU})
  on the weights of the approximating networks,
  one can show using Remark~\ref{rem:TheRemarkForTheConverse}
  that the approximation rate from above is \emph{optimal}
  {up to a factor that is logarithmic in $1/\eps$}.
  Since this encodability condition is quite technical,
  however, we do not state this result in detail.
\end{rem*}

\begin{proof}
  Let ${f := (F)_1 : \R^{\FirstN{C_0} \times \Gr} \to \R}$
  and $L := (2 + \lceil \log_2 \beta \rceil_+) \cdot (11 + \beta / D)$,
  where $D := C_0 \cdot |\Gr|$.
  With the constant $c_0 = c_0 (D,\beta) > 0$ provided by
  \cite[Theorem 3.1]{PetV2018OptApproxReLU}, we see that there is an FNN
  $\Phi_\eps^f$ with at most $L$ layers, and such that
  ${W(\Phi_\eps^f) \leq c_0 \cdot (|\Gr|^{-1/p} \cdot \eps)^{-D / \beta}}$ and
  \(
    \|
      R_\varrho(\Phi_\eps^f) - f
    \|_{L^p ([-{1}/{2}, {1}/{2}]^{\FirstN{C_0} \times \Gr})}
    \leq |\Gr|^{-1/p} \cdot \eps
  \).

  Thus, setting $c := 2 c_0 \cdot |\Gr|^{D / (\beta p)}$, Theorem~\ref{thm:main} yields a CNN
  $\Psi_\eps^F$ satisfying all the stated properties.
\end{proof}

For the sake of brevity, we refrain from explicitly stating the CNN versions of the results in
\cite{Barron1993,
      bolcskei2017optimal,
      Cybenko1989,
      Maiorov1999LowerBounds,
      Mhaskar:1996:NNO:1362203.1362213,
      Mhaskar1993,
      YAROTSKY2017103}.

{Finally, we remark that Theorem~\ref{thm:main} yields a new proof of the
\emph{universal approximation theorem for CNNs}
that was originally derived by Yarotsky \cite{YarotskyWasFirstAgain}.
This theorem states that if $F : \R^{\FirstN{C_0} \times \Gr} \to \R^{\FirstN{N} \times \Gr}$ is
continuous and translation equivariant, and if $\Omega \subset \R^{\FirstN{C_0} \times \Gr}$
is $\Gr$-invariant and compact, then $F$ can be uniformly approximated on $\Omega$ by
$\varrho$-realisations of CNNs of any fixed depth $L \geq 2$,
as long as $\varrho$ is continuous, but not a polynomial.}

\section*{Acknowledgments}
P.P.~is grateful for the hospitality of the
Katholische Universität Eichst\"att--Ingolstadt
where this problem was formulated and solved during his visit. P.P is supported by a DFG research fellowship.

\bibliographystyle{abbrv}
\bibliography{ref}

\end{document}

%% file: Nutshell.tex
Let us give a simplified but honest description of our results for the special setting of
CNNs acting on images.
In this setting, we think of the inputs of our networks as being square matrices,
i.e., living in $\R^{d \times d}$, $d \in \N$.


\subsubsection*{Convolutions and the relation to groups}
At least in the mathematical literature, the convolution on $\R^{d \times d}$ is typically defined by
\begin{equation}
  (X \ast Y)_j
  := \sum_{i\in \{0,\dots, d-1\}^2}
       X_i \, Y_{(j-i) \  \mathrm{ mod } \ d}
  \label{eq:IntroConvolution}
\end{equation}
for $X,Y \in  \R^{d \times d}$ and $j \in \{0,\dots, d-1\}^2$.
This convolution is induced by the group structure of $(\Z / d \Z)^2$.
Precisely, if we identify an ``image'' ${X = (X_{i,j})_{i,j \in \{0,\dots,d-1\}}}$ with the function
$F_X : (\Z / d \Z)^2 \to \R, (i + d \Z, j + d \Z) \! \mapsto \! X_{i,j}$, 
then we have ${F_X \ast F_Y = F_{X \ast Y}}$, where
\[
  (F \ast G)(g) = \!\! \sum_{h \in (\Z / d \Z)^2} \!\! F(h) \, G(g - h)
  \quad \text{for} \quad g \in (\Z / d \Z)^2 \text{ and } F,G : (\Z / d \Z)^2 \to \R.
\]
We remark that in the computer science literature, also other conventions for the convolution
are commonly used; for instance, one can use \emph{zero padding} instead of the periodic boundary handling
in Equation~\eqref{eq:IntroConvolution}.

In the main body of the paper, we will state our results for convolutions
stemming from general finite groups, which has the advantage that the results
apply to arbitrary input dimensions,
and thus also for sound signals or videos in addition to images.
Second, working on general groups simplifies the notation.

\subsubsection*{Fully connected and convolutional neural networks}
We define FNNs as certain functions that result from repeatedly applying affine-linear maps
and a non-linear componentwise operation.
More precisely, an FNN $\Phi$ with $L$ layers is of the form $\Phi(x) = (F_L \circ \cdots \circ F_1)(x)$,
where each $F_i$ is given by ${F_i (x) = \varrho (A_i \, x + b_i)}$; here, the non-linearity
$\varrho : \R \to \R$ is applied componentwise.
The number of \emph{active parameters} of an FNN is the number of active parameters
of the involved affine-linear maps; that is, the total number of nonzero entries of the matrices
$A_i$ and the bias vectors $b_i$.

For CNNs, we use the same definition as in \cite{YarotskyWasFirstAgain},
which formalizes the more intuitive description in \cite[Chapter 9.5, Equation (9.7)]{Goodfellow-et-al-2016}.
As for FNNs, a CNN $\Phi$ with $L$ layers is of the form $\Phi(X) = (F_L \circ \cdots \circ F_1)(X)$,
where the computation of each layer $F_\ell$ is as shown in Figure~\ref{fig:CNNExplanation}.
More formally, each layer takes as input a ``stack of images'' $X = (X_1,\dots,X_N)$,
where each image $X_i \in \R^{d \times d}$ is called an \emph{input channel}.
The output of the layer is again a ``stack of images'' $Y = (Y_1,\dots,Y_M)$
which is computed as follows:
\begin{enumerate}
  \item Several convolution kernels $c_1, \dots, c_k \in \R^{d \times d}$
        are applied to the individual input channels; this results in
        \[
          (V_{i,\ell})_{i \in \{1,\dots,k\}, \ell \in \{1,\dots,N\}}
          = (c_i \ast X_\ell)_{i \in \{1,\dots,k\}, \ell \in \{1,\dots,N\}}.
        \]

  \item Different linear combinations of the channels $V_{i,\ell}$ are formed;
        this results in the ``stack of images'' $(W_{j})_{j \in \{1,\dots,M\}}$, where
        \(
          W_j = \sum_{i = 1}^k
                  \sum_{\ell = 1}^N
                    A_{j,(i,\ell)}
                    V_{i,\ell}
        \)
        for certain coefficients $A_{j,(i,\ell)} \in \R$.
        \vspace{0.1cm}

  \item The individual channels $Y_j \in \R^{d \times d}$ of the output of the layer are given by
        $Y_j = \varrho (W_j + b_j)$, where $b_j \in \R$,
        and where the addition $W_j + b_j$ as well as the application
        of the \emph{activation function} $\varrho$ are componentwise.
\end{enumerate}
Overall, a CNN $\Phi$ thus computes a function
\(
  \Phi : (\R^{d \times d})^{C_{\mathrm{in}}} \to (\R^{d \times d})^{C_{\mathrm{out}}}
\),
where $C_{\mathrm{in}}$ and $C_{\mathrm{out}}$ denote the number of input or output channels,
respectively.

The number of \emph{active parameters} of a CNN is the total number of all non-zero entries
of the convolution kernels plus the total number of non-zero coefficients used
for forming the linear combination of the channels.

\begin{figure}[h]
  \begin{center}
    \includegraphics[width=\textwidth]{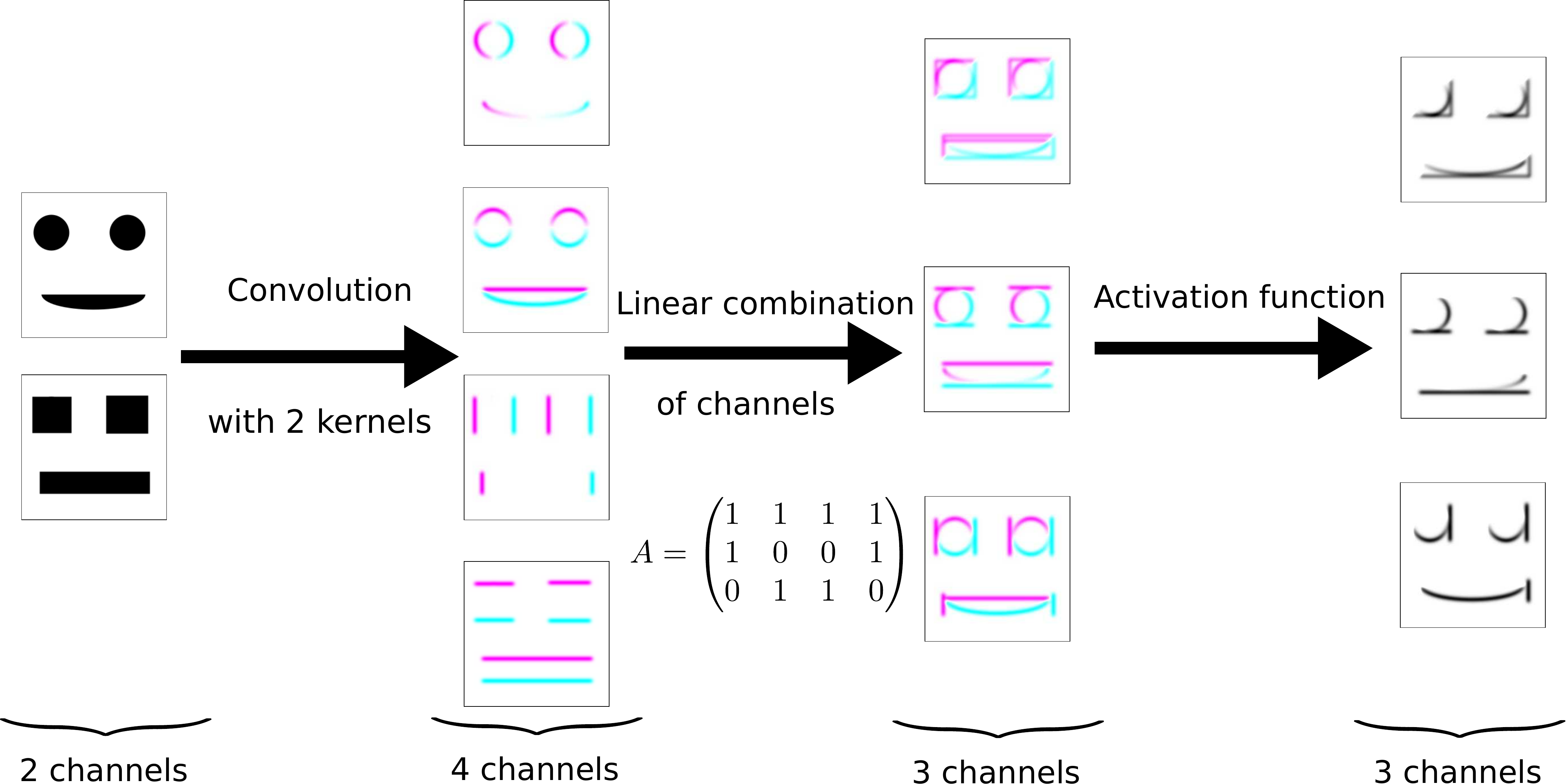}
  \end{center}
  \captionsetup{justification=centering}
  \caption{\label{fig:CNNExplanation}
           Illustration of the computation performed by each layer of a CNN.
           In the middle two rows, positive values are shown in turquoise,
           while negative values are shown in purple.
           The first convolution kernel picks up vertical edges,
           while the second picks up horizontal edges.
           In the notation of the text, we have $N = 2$, $k = 2$, and $M = 3$.}
\end{figure}

\subsubsection*{Translation equivariant maps}


Convolutions are closely related to \emph{translations}.
Precisely, for $i \in \{0,\dots,d-1\}^2$ and $X \in \R^{d \times d}$, we denote by
$S_i X$ the translation of $X$ by $i$, given by $(S_i X)_j = X_{j - i \ \mathrm{mod} \ d}$;
that is, the boundaries are treated periodically.
We then have $S_i (X \ast Y) = (S_i X) \ast Y = X \ast (S_i Y)$ for all $X,Y \in \R^{d \times d}$.

Now, given a stack of images $X = (X_1,\dots,X_N)$, write $S_i X = (S_i X_1, \dots, S_i X_N)$.
We say that a function $F : (\R^{d \times d})^{N} \to (\R^{d \times d})^{M}$ is
\emph{translation equivariant} if it commutes with translations, that is, if
\(
  F(S_i X)
  = S_i [F(X)]
\)
for all $X \in (\R^{d \times d})^N$.
It is not hard to see that every CNN defines a translation equivariant map.

An important property of translation equivariant maps is that they are uniquely determined
by each of their \emph{components}.
Indeed, for $i \in \{0,\dots,d-1\}^2$ and ${X = (X^{(1)},\dots,X^{(N)}) \in (\R^{d \times d})^N}$,
let us write $\pi_i (X) = (X^{(1)}_i, \dots, X^{(N)}_i) \in \R^N$ for the projection of
(each channel of) $X$ onto the $i$-th component.
Note that $\pi_i (X) = \pi_0 (S_{-i} X)$.
Therefore, if $F : (\R^{d \times d})^N \to (\R^{d \times d})^M$ is translation equivariant, then
\[
  \pi_i (F(X))
  = \pi_0 \big( S_{-i} [F(X)] \big)
  = \pi_0 \big( F(S_{-i} X) \big) .
\]
Thus, if the component $\pi_0 \circ F$ is known, then so are all other components
$\pi_i \circ F$.


\subsubsection*{The equivalence between approximation properties of FNNs and CNNs}

We next state our main results, Theorem~\ref{thm:main} and Remark~\ref{rem:TheRemarkForTheConverse},
in the simplified setting of CNNs acting on images $X \in \R^{d \times d}$.
This result holds for any measurable activation function $\varrho : \R \to \R$.


\begin{theorem}\label{thm:MainIntro}
  Let $F: (\R^{d\times d})^{N} \to \R^{d\times d}$ be translation equivariant.
  Furthermore, let ${\Omega_1,\dots,\Omega_N \subset \R}$ be measurable, and set
  \(
    \Omega := \Omega_1^{d \times d} \times \cdots \times \Omega_N^{d \times d}
           \subset (\R^{d \times d})^{N}.
  \)
  Then, the following hold for each $p \in (0,\infty]$, each $\eps > 0$, and in fact also for $\eps = 0$:
  \begin{enumerate}
   \item If there is an FNN $\Phi$ with $W \in \N$ active parameters and
         $L$ layers satisfying ${\|\Phi - \pi_0 \circ F\|_{L^p(\Omega)} \leq \eps}$,
         then there is a CNN $\Psi$ with $L$ layers and at most $2W$ active parameters and such that
         $\|\Psi - F\|_{L^p(\Omega)} \leq d^{2/p} \, \eps$.
         \vspace{0.1cm}

   \item If there exists a CNN $\Psi$ with $W$ active parameters and $L$ layers satisfying
         $\|\Psi - F\|_{L^p(\Omega)} \leq \eps$,
         then there exists an FNN $\Phi$ with $L$ layers and at most $d^4 W$ weights and such that
         $\|\Phi - \pi_0 \circ F\|_{L^p(\Omega)} \leq \eps$.
  \end{enumerate}
\end{theorem}
\begin{rem*}
  1) For simplicity, we here only consider functions $F$ with one output channel.
  The theorem also holds in greater generality.

  2) In addition to the bounds regarding the number of active parameters, Theorem~\ref{thm:main}
  also provides bounds on the number of neurons.
\end{rem*}

\subsubsection*{The proof idea}

Part~(2) of Theorem~\ref{thm:MainIntro} is almost immediate,
because convolutions are special linear maps and therefore every CNN is also an FNN.
The increase from $W$ active parameters to $d^4 W$ active parameters is due to the
distinct way in which the number of active parameters are counted for CNNs and FNNs.
For instance, a convolution $\R^d \ni x \mapsto x \ast a \in \R^d$ with a vector $a \in \R^d$
has at most $d$ active parameters, while the matrix $A_a \in \R^{d \times d}$
representing the map $x \mapsto x \ast a$ has up to $d^2$ active parameters.
This is explained in more detail in Remark~\ref{rem:TheRemarkForTheConverse}.

\medskip{}

The first part of Theorem~\ref{thm:MainIntro}, on the other hand, is more intricate.
The proof idea is as follows:
The FNN $\Phi$ is of the form
$\Phi = T_L \circ \varrho \circ T_{L-1} \circ \cdots \circ \varrho \circ T_2 \circ \varrho \circ T_1$,
where $\varrho$ is applied componentwise, and where $T_1 : (\R^{d \times d})^N \to \R^{N_1}$
and $T_\ell : \R^{N_{\ell - 1}} \to \R^{N_\ell}$ ($2 \leq \ell \leq L$) are affine-linear;
here, $N_\ell$ is the number of neurons in the $\ell$-th layer of $\Phi$.
The general idea is to construct from this a ``lifted'' CNN $\Psi$ in which
each neuron (representing a real number $x \in \R$) of the FNN $\Phi$ is replaced by a full
\emph{channel} (representing an image $X \in \R^{d \times d}$).
This needs to be done in such a way that the lifted CNN $\Psi$ computes a
``translation equivariant version'' of the FNN $\Phi$.

Formally, we will construct affine-linear maps
$A_1 : (\R^{d \times d})^N \to (\R^{d \times d})^{N_1}$
and $A_\ell : (\R^{d \times d})^{N_{\ell - 1}} \to (\R^{d \times d})^{N_\ell}$ ($2 \leq \ell \leq L$)
which have two important properties.
First, all of the maps $A_\ell$ ($1 \leq \ell \leq L$) are convolutional
(corresponding to the first two steps in Figure~\ref{fig:CNNExplanation}).
Second, we have
\begin{equation}
  \pi_0 \circ A_1 = T_1
  \quad \text{and} \quad
  \pi_i \circ A_\ell = T_\ell \circ \pi_i
  \quad \text{for all} \quad i \in \{0,\dots,d-1\}^2.
  \label{eq:IntroMainProofIdea}
\end{equation}
This means that projecting each channel of $A_\ell (X)$ onto the $i$-th component is the same
as projecting each channel of $X$ onto the $i$-th component and then applying $T_\ell$.
Denoting by $\Psi$ the CNN defined by the maps $A_1,\dots,A_L$,
Equation~\eqref{eq:IntroMainProofIdea} shows
\[
  \pi_0 \circ \Psi
  = \pi_0 \circ A_L \circ \varrho \circ A_{L-1} \circ \cdots \circ \varrho \circ A_1
  = T_L \circ \varrho \circ T_{L-1} \circ \cdots \circ \varrho \circ T_1
  = \Phi.
\]
Now, since $\pi_0 \circ \Psi = \Phi$ is close to $\pi_0 \circ F$, and since both $F$ and $\Psi$
are translation equivariant, it follows that $\Psi$ is close to $F$.
We refer to the proof of Theorem~\ref{thm:main} for the precise construction of the maps $A_1,\dots,A_L$.

%% file: RelatedWork.tex
The findings in \cite{UniConv,ZhouDeepDistributedConvolutionalNetworks}
are closely related to the results presented in this paper.
The focus in \cite{UniConv,ZhouDeepDistributedConvolutionalNetworks}, however,
is specifically on CNNs that employ $1$-dimensional convolutions
\emph{with very short convolution kernels} of a fixed length $2 \leq s \leq d$.
For convolutional networks of this type, \cite[Theorem A]{UniConv} establishes a universal
approximation result:
Given a compact set $\Omega \subset \R^d$, a sparsity parameter $2 \leq s \leq d$,
an error bound $\eps > 0$, and an arbitrary continuous function $f : \Omega \to \R$,
there is a depth $J = J(f,\eps) \in \N$ and a CNN $\Phi_{f,\eps}$ as just described
and of depth $J$ such that $\| f - \Phi_{f,\eps}\|_{C(\Omega)} \leq \eps$.

The main differences between our results and the findings in \cite{UniConv} are the following:
First, the emphasis in \cite{UniConv} is on CNNs using convolution kernels with small support,
while we put no restriction on the size of the convolution kernels.
Second, the type of networks considered in \cite{UniConv} is quite different from the ones considered by us:
In \cite{UniConv}, each layer of the network performs \emph{a single convolution} before the non-linearity.
In contrast, each layer in our CNNs performs multiple convolutions and can have multiple
input- and output channels.
Further, the networks in \cite{UniConv} do \emph{not} produce translation equivariant functions,
since the last layer is allowed to be non-convolutional, and since the networks employ
\emph{bias vectors} instead of one constant bias for each channel of the network.
For this reason, the networks considered in \cite{UniConv} can be universal
in the class of \emph{all} continuous functions, not just in the class of translation-equivariant ones.
Third, the universality results and the approximation rates in \cite{UniConv}
\emph{require the number of layers to grow unboundedly}.
In contrast, we consider networks of a fixed depth.
Finally, \cite{UniConv} only considers \emph{one-dimensional convolutions}
(corresponding to data like sound signals), while our results apply for arbitrary dimensions,
and thus to input signals like images or videos.

In \cite[Proposition 3.1]{YarotskyWasFirstAgain}, a universal approximation theorem for CNNs is established.
Specifically, it is shown there that every equivariant continuous function can be
approximated arbitrarily well by a CNN with one hidden layer.
The definition of CNNs in \cite{YarotskyWasFirstAgain} uses convolutions based on general
(compact, abelian) groups and coincides with the one in the present paper, when specialised to finite groups.
In contrast to the present work, \cite{YarotskyWasFirstAgain} does not study rates of approximation,
but universality only.
In addition, in \cite[Section 3]{YarotskyWasFirstAgain},
Yarotsky establishes a range of more abstract universality results
for networks that take as inputs functions $f \in L^2(\R^d)$ instead of
the discrete inputs that are considered in the present paper.
For these types of networks, the role of pooling is analysed in \cite[Section 4]{YarotskyWasFirstAgain}.

We would also like to mention the paper \cite{MallatUnderstandingDeepCNN} which---while not
concerned with approximation properties---is one of the first rigorous investigations
of the mathematical properties of CNNs.
